\documentclass[a4paper,notitlepage,twoside,leqno,12pt]{amsart}

\usepackage{bbm,pifont,latexsym}
\usepackage{anysize}
\marginsize{3.6cm}{3.6cm}{3.6cm}{3.6cm}

\usepackage{dcolumn,indentfirst}
\usepackage[colorlinks=false,linkcolor=black,citecolor=black,urlcolor=black,bookmarksopen=true,linktocpage=true,pdfstartview=FitH]{hyperref}
\usepackage{amsmath,amssymb,amscd,amsthm,amsfonts,mathrsfs}
\usepackage{color,graphicx,xcolor,graphics,subfigure,extarrows,caption2}
\usepackage{titletoc}

\newtheorem{thm}{Theorem}[section]
\newtheorem*{main}{Main Theorem}
\newtheorem{cor}[thm]{Corollary}

\newtheorem{lema}[thm]{Lemma}

\theoremstyle{definition}
\newtheorem*{defi}{Definition}

\newtheorem{que}{Question}

\usepackage{enumerate,enumitem}

\makeatletter\@addtoreset{equation}{section}\makeatother

\begin{document}

\author{Gaofei Zhang}
\address{Department of Mathematics, QuFu Normal University, Qufu  273165,   P. R. China}
\email{zhanggf@hotmail.com}

\title{Jordan mating is always possible for polynomials}

\begin{abstract}
Suppose $f$ and $g$ are two post-critically finite
polynomials of degree $d_1$ and $d_2$ respectively  and suppose     both of them have a finite super-attracting fixed point  of  degree $d_0$.    We prove  that one can always construct  a rational map $R$ of degree
 $$
 D = d_1 + d_2 - d_0
 $$ by gluing
 $f$ and $g$   along the Jordan curve boundaries  of the  immediate super-attracting basins. The result can be used to construct many rational maps with interesting dynamics.

\end{abstract}

\subjclass[2010]{Primary: 37F45; Secondary: 37F10, 37F30}

\keywords{}

\date{\today}



\maketitle


\section{Introduction}

Polynomial mating was an operation  proposed by Douady and Hubbard to understand  the  dynamics of rational maps.  Very roughly speaking, for two post-critically finite polynomials $P$ and $Q$ of degree $d \ge 2$ with both the Julia sets being connected, we may glue  $f$ and $g$ along the Julia sets to get a topological map $F$.
  We say $f$ and $g$  are   matable if
$F$ is a branched covering map of the two sphere to itself,  and moreover,   $F$ is topologically conjugate to some rational map.
Noting that the Julia set is the boundary of the immediate super-attracting basin of the infinity, the idea can be naturally extended  to the situation of rational maps. Suppose    $f$ and $g$  are two post-critically finite rational maps both of which have a simply connected
immediate super-attracting basin of degree $d_0 \ge 2$  such that
 there are no other critical orbits which intersect the immediate basins.  Then  one may construct  a topological map by
  gluing
   $f$ and $g$ along   the attracting basin boundaries  and  then copy this gluing for all the pre-images of the attracting basins.  As in the case of polynomial mating,     we  say $f$ and $g$ are matable if $F$ is a branched covering of the two sphere to itself,  and moreover, $F$ is topologically conjugate  to some rational map $G$.

A particularly important case is that
both the super-attracting basins are Jordan domains (Noting that all bounded immediate attracting basins of polynomials are Jordan domain \cite{PY}).    In this case,   no pinching happens when gluing $f$ and $g$ along the Jordan boundary and  the topological map is always a branched covering of the two sphere to itself.  Let us describe this topological construction as follows.  Let $D_f$ and $D_g$ denote the two Jordan super-attracting basins and $D_f^c, D_g^c$ be there complements respectively.  Let $\phi: D_f \to \Delta$ and $\psi: D_g \to \Delta$ be the holomorphic isomorphism which conjugate $f$ and $g$ to $z \mapsto z^{d_0}$. Then for each $1 \le k \le d_0-1$,
\begin{equation} \label{m} \Phi = \phi^{-1}  \bigg{(}\frac{e^{2 k \pi i/(d-1)}} {\psi}\bigg{)}: \partial D_g \to \partial D_f\end{equation}   is a homeomorphism which reverses  the orientation. We can extend it to a homeomorphism of the sphere so that it
 maps   $\overline{D_g}$ to  $D_f^c$ and maps $ D_g^c$ to   $\overline{D_f}$.  Now we glue     $D_g^c$  and $D_f^c$ by identifying the points $x$ and $\Phi(x)$. It is clear that $D_g^c \bigsqcup_{x \sim \Phi(x)} D_f^c$ is a topological two sphere. Define
$$
F:  D_g^c \bigsqcup_{x \sim \Phi(x)} D_f^c \to D_g^c \bigsqcup_{x \sim \Phi(x)} D_f^c
$$ by setting
 \begin{equation}\label{oo}
F(z)  =
\begin{cases}
f(z)  & \text{ for $z \in D_f^c$ and $f(z) \in D_f^c$}, \\
\Phi^{-1}\circ f(z)  & \text{ for $z \in D_f^c$ and $f(z) \in D_f$}, \\
g(z)  & \text{ for $z \in D_g^c$ and $f(z) \in D_g^c$}, \\
\Phi\circ g(z)  & \text{ for $z \in D_g^c$ and $f(z) \in D_g$}.
\end{cases}
\end{equation}

Since by assumption  no other critical orbits of $f$ and $g$
 enter into $D_f$ and $D_g$ respectively,
 the way of extending $\Phi: \partial D_g \to \partial D_f$  dose not affect the combinatorially equivalent class of $F$. In the case that $F$ has no Thurston obstruction, we have a rational map $G$ which is combinatorially equivalent to $F$ (One can actually prove that $G$ is topologically conjugate to $F$).    Unlike the usual mating, whose Julia sets is the disjoint union of   $J_f$ and $J_g$ with  those points in a ray equivalent class being identified,
 the Julia set of $G$ contains infinitely many copies of $J_f$ and $J_g$.  To get $J_G$,  one may start  from $J_g \bigsqcup_{x \sim \Phi(x)}  J_f$, and then  iteratively  fill the pre-images of $D_f$  and $D_g$ by   copies of $J_g$  and $J_f$     respectively.    To make a distinction with  the usual mating,  we call such mating a
  $\emph{Jordan mating}$.

 In contrast to the usual mating,  for which   there exist  cubic polynomials  which are topologically matable but not matable \cite{ST},    Jordan mating  is always possible for two  polynomials. We will actually prove a stronger result.
 \begin{defi}
 For $d_0 \ge 2$,
let   $\mathcal{R}_{d_0}$ denote the family of all post-critically finite rational maps which have  a  marked  immediate
 super-attracting  basin $D$    which is  a Jordan domain and of degree $d_0$ such that   all the other critical orbits do not intersect $D$.
 \end{defi}
\begin{main}
Let $d_0 \ge 2$.  Suppose $f, g \in \mathcal{R}_{d_0}$ such that at least one of them is a polynomial.
 Then $f$ and $g$  can be mated  into a rational map $R$ of degree $$D = d_1 + d_2 -d_0$$  with $d_1$ and $d_2$ being the degrees of $f$ and $g$ respectively. In particular, Jordan mating is always possible for   polynomials.
\end{main}

   Since the topological map $F$
   in our case is always a branched covering of the sphere to itself, all we need to do is to show that $F$ has no Thurston obstructions.   up to now there is no general way to check if a given topological map has Thurston obstructions or not, although many tools and ideas haven been developed\cite{BFH}\cite{Re}\cite{ST}\cite{Tan}\cite{T}.  The   idea of our   proof is to associate each non-peripheral curve a quantity which is monotonically increasing   as we iterate the topological map. This property will lead us to get a Levy cycle from an irreducible Thurston obstruction.  We then show that such a Levy cycle can be deformed into a Levy cycle of the rational map $g$, which is a contradiction. Our argument relies essentially on   the assumption that one of the two rational maps is a polynomial.

 \begin{que}
Is the Jordan mating  always  possible for      rational maps in $\mathcal{R}_{d_0}$?
\end{que}

\section{Examples}

 In this section we give two examples of Jordan mating.
 Let $f$  be a cubic polynomial   which has a degree two super-attracting fixed point at the origin    so that
the other finite critical point  $c$    belongs to the boundary of the super-attracting basin, and moreover,  $f^2(c) = f(c)$.
Let $g$  be a cubic polynomial   which has a a degree two super-attracting fixed point at the origin   so that
the other finite critical point   $c$   belongs to the boundary of the super-attracting basin, and moreover,  $g^3(c) = g^{2}(c) \ne g(c)$.
Let $h$ be a post-critically finite cubic rational map so that it has a degree two super-attracting fixed point at $\infty$  and the Julia set is a Sierpinski carpet.

\begin{figure}[!htpb]
  \setlength{\unitlength}{1mm}
  \begin{center}

 \includegraphics[width=70mm]{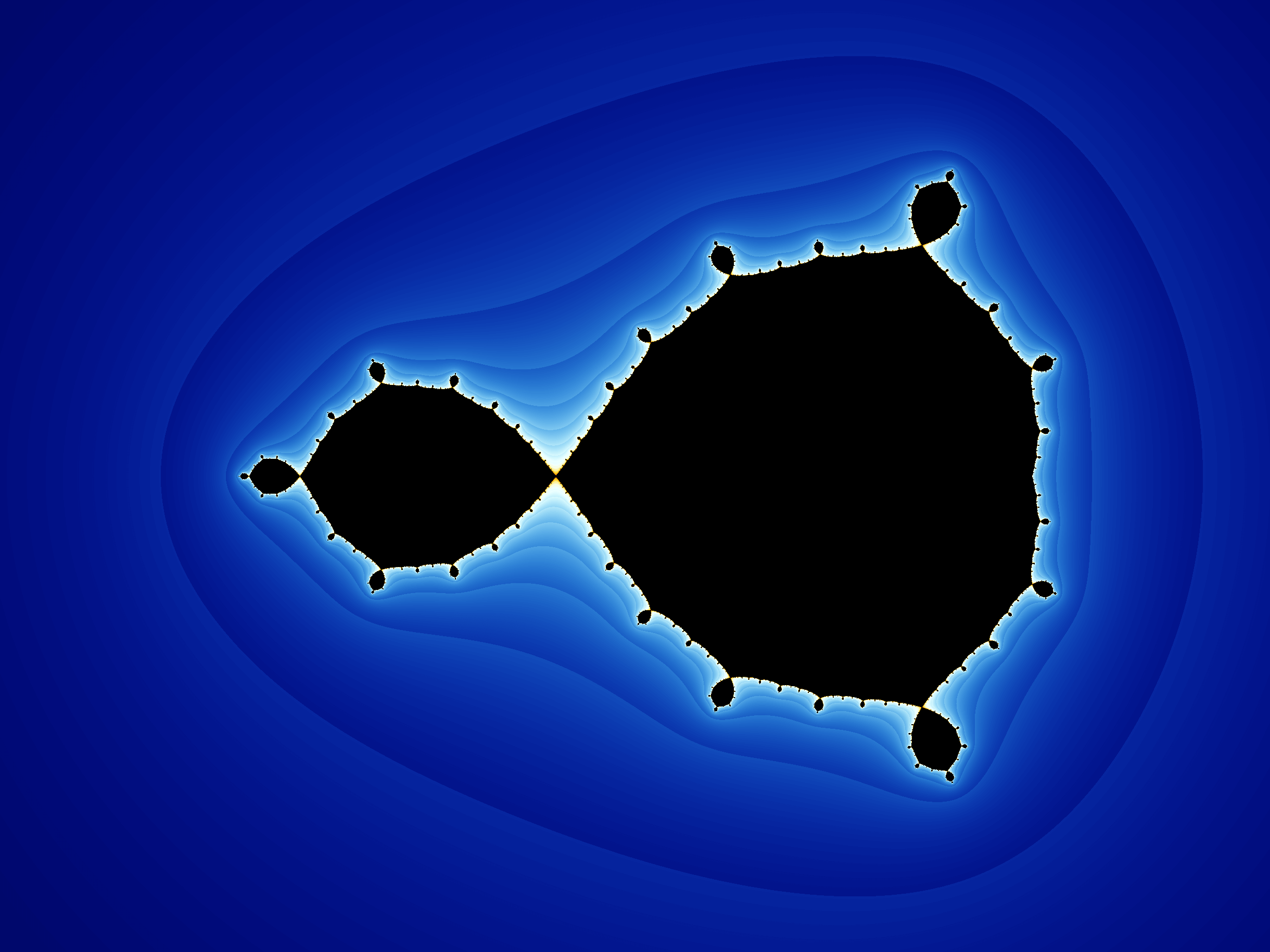}
  \caption{The Julia set for $f$}
  \label{Figure-1}

  \end{center}
  \end{figure}

  \begin{figure}[!htpb]
  \setlength{\unitlength}{1mm}
  \centering
  \includegraphics[width=70mm]{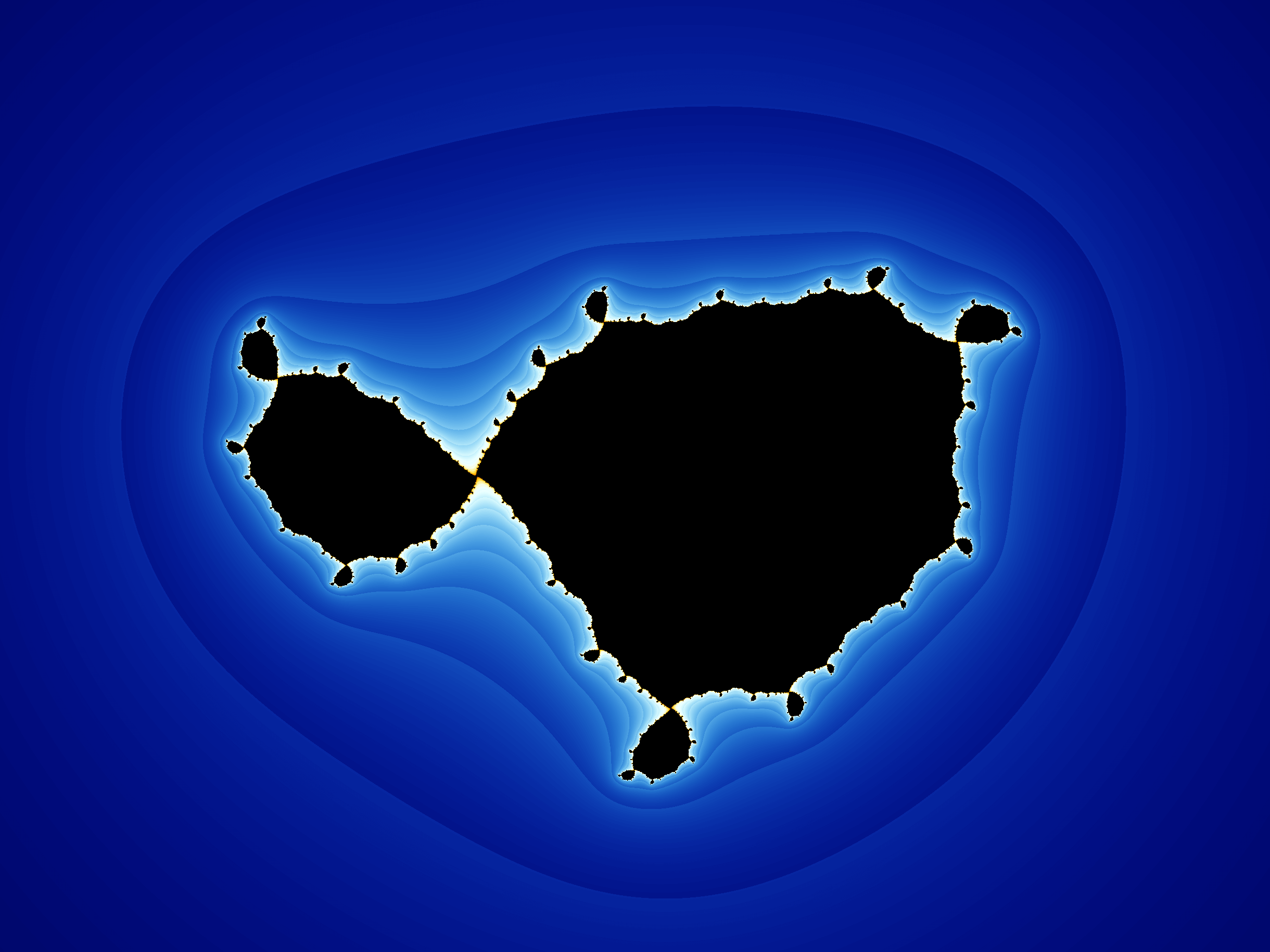}
  \caption{The Julia set for $g$}
  \label{Figure-1}
  \end{figure}

  \begin{figure}[!htpb]
  \setlength{\unitlength}{1mm}
  \begin{center}

 \includegraphics[width=70mm]{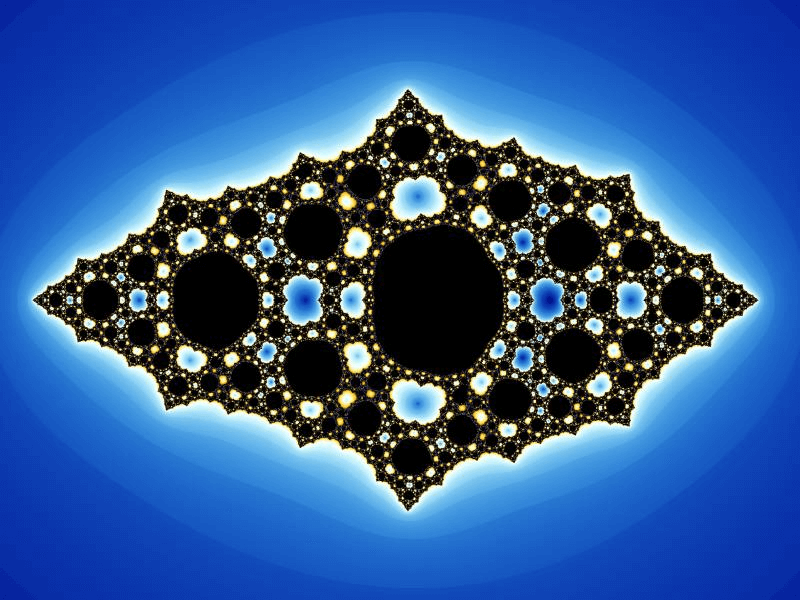}
  \caption{The Julia set for $h$}
  \label{Figure-1}

  \end{center}
  \end{figure}

  \begin{figure}[!htpb]
  \setlength{\unitlength}{1mm}
  \centering
  \includegraphics[width=70mm]{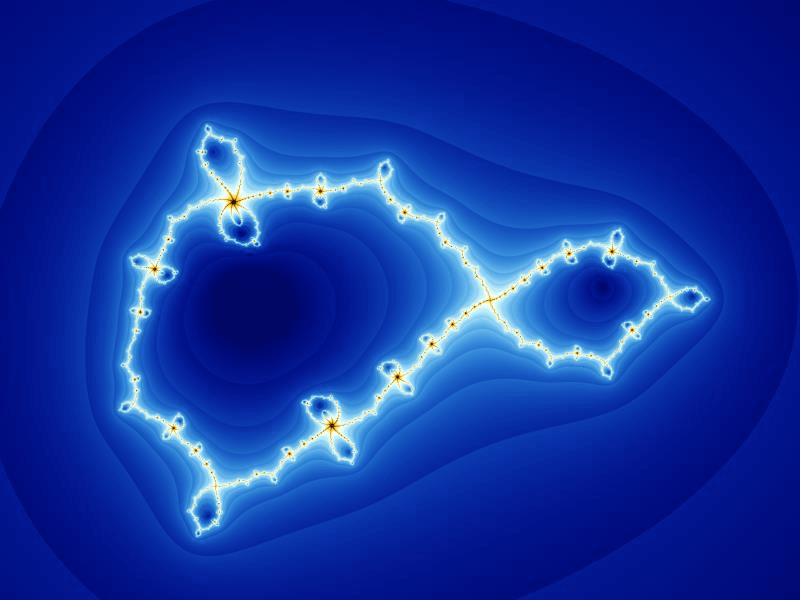}
  \caption{Jordan mating of $f$ and $g$}
  \label{Figure-1}
  \end{figure}

    \begin{figure}[!htpb]
  \setlength{\unitlength}{1mm}
  \begin{center}

 \includegraphics[width=70mm]{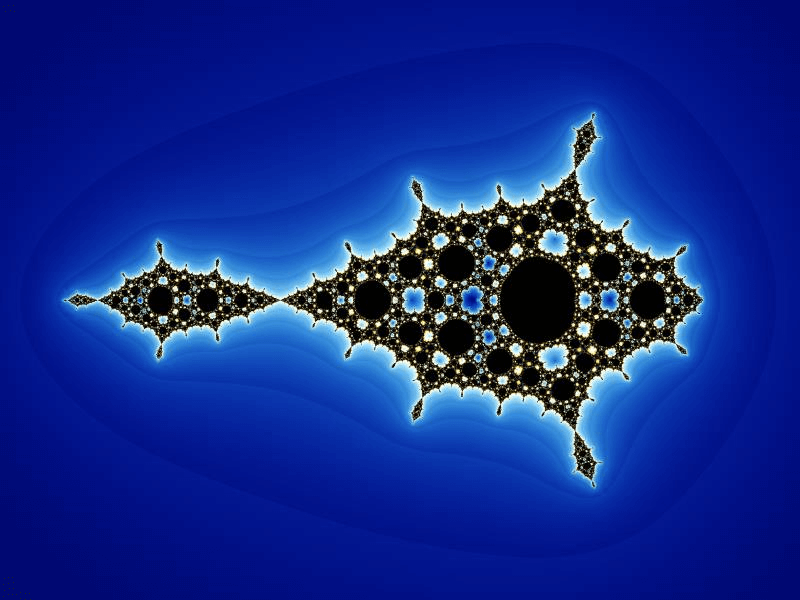}
  \caption{Jordan mating of  $f$ and $h$}
  \label{Figure-1}

  \end{center}
  \end{figure}

\section{Proof of the main theorem} The reader may refer  to \cite{DH} \cite{Pi} for the details of the Thurston's theory for characterization of post-critically finite rational maps. Throughout the paper we use $\widehat{\Bbb C}$, $\Bbb C$, $\Bbb C^*$,  $\Bbb T$ and  $\Bbb D$
to denote the Riemann sphere, the complex plane, the puncture complex plane, the unit circle and the unit disk respectively.
Assume that $f, g \in \mathcal{R}_{d_0}$  with $f$ being  a polynomial. Then there are $d_0 -1$ ways to glue $f$ and $g$ along the boundary of the marked attracting basin, see (\ref{m}).
Let $F$ be one of such topological maps.  All we need to do is to show that $F$ has no Thurston obstructions.  We may identify the boundaries  of the two  marked immediate attracting basins with   $\Bbb T$.    We  may  assume that  $F: \Bbb T \to \Bbb T$ is given by $z \mapsto z^{d_0}$, and    up to combinatorial equivalence,  $F = f$ outside $\Bbb T$ and $F = g$ inside $\Bbb T$.   When a post-critical point $x \in P_F$ belongs to  the forward orbit of some critical point of $f$, we  write $x \in P_f$, and similarly, if   it belongs to  the forward orbit of some critical point of $g$, we  write $x \in P_g$.    In particular, we have
$$
P_F  = P_f \cup P_g.
$$

Suppose $\gamma$ is a non-peripheral curve of  $F$.  In the following we only concern those $\gamma$ so that $\gamma \cap \Bbb T \ne \emptyset$.  By homotopy rel $P_F$ we may  assume that $\gamma \cap \Bbb T$ is a finite set.
Then  $\gamma - \Bbb T$  consists of finitely many  curve segments. Let $\sigma$ be any of such curve segments.
We call  $\sigma$  of $\emph{type P}$ (polynomial type)     if  it  is outside $\Bbb T$,  otherwise, we call it of $\emph{type R}$ (rational type). Let $I \subset \Bbb T$   be the arc so that $\partial I = \partial \sigma$ and $I$ is homotopic to $\sigma$ rel $\partial \sigma$ in $\Bbb C^*$. Let $D(\sigma)$ be the union of $I$ and
the bounded domain bounded by $\sigma$ and $I$.

\begin{figure}[!htpb]
  \setlength{\unitlength}{1mm}
  \centering
  \includegraphics[width=100mm]{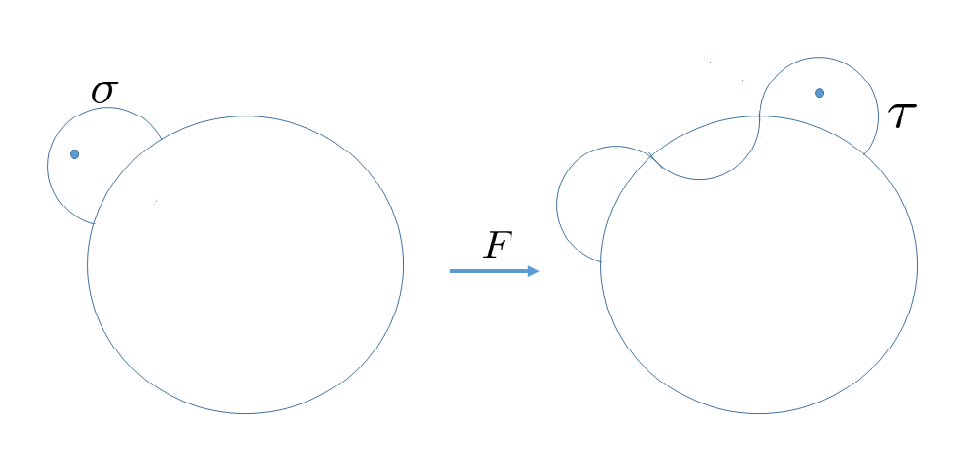}
  \caption{curve segments of polynomial type}
  \label{Figure-1}
  \end{figure}

\begin{lema}\label{ess}
Suppose $\gamma$ is non-peripheral curve in $\widehat{\Bbb C} - P_F$ and $\eta$ is a non-peripheral component of $F^{-1}(\gamma)$.  Suppose $\sigma$ is  a \emph{type P} curve segment of $\eta $ and
$x \in D(\sigma) \cap P_F$.  Then there is some \emph{type P} curve segment   $\tau$ of $\gamma$ such that
 $F(x) \in D(\tau)$.
\end{lema}

\begin{proof}
By assumption the orbit of the critical points of $F$ which belongs to $\widehat{\Bbb C} \setminus \Bbb D$ does not enter $\Bbb D$. So $F(x) \notin \Bbb D$. Since the action of $F$ on the outside of $\Bbb T$ is given by the polynomial $f$, the image of $D(\sigma)$ is bounded whose boundaries is a subset of the union of  $\Bbb T$ and finitely many curve segments of $\emph{type P}$ and $\emph{R}$ of $\gamma$.  Since $F(x) \notin \Bbb D$, it follows that there is some $\emph{type P}$
curve segment $\tau$  of $\gamma $ so that $F(x) \in D(\tau)$.
\end{proof}

For $x \in P_f$ and $\gamma$ a non-peripheral curve in $\widehat{\Bbb C} -  P_F$, let $\Sigma_x(\gamma)$ denote the set of all the $\emph{type P}$ curve segments  $\sigma$  of $\gamma$ so that $x \in D(\sigma)$. Let $N(\Sigma_x(\gamma))$ denote the number of the elements in $\Sigma_x(\gamma)$. Let
$$
N(\Sigma_x([\gamma])) = \min_{\gamma'} N(\Sigma_x(\gamma'))
$$ where $\min$ is taken over all non-peripheral curves $\gamma'$ which are homotopic to $\gamma$ in $\widehat{\Bbb C} - P_F$. We need more notations.
\begin{itemize}
\item Let $\mathcal{O}$ denote the set of periodic points in $P_f$.
\item Let $\Sigma$  denote the class of non-peripheral curves $\gamma$
 so that there is $x \in \mathcal{O}$ with $N(\Sigma_x([\gamma])) > 0$    and    let $\Pi$ denote the class of other non-peripheral curves.  \item Let $\Lambda$  denote the class of non-peripheral curves $\gamma$ so that $\Sigma_x([\gamma]) = 0 $
  holds for any $x \in P_f$.
  \end{itemize}

\begin{lema}\label{ev}
There is some $n$ large such that for any $\gamma \in \Pi$,  if $\eta$ is a non-peripheral component of $F^{-n}(\gamma)$, then   $\eta \in \Lambda$.
\end{lema}

\begin{proof} Since every critical point of $f$  is eventually periodic,
we have an integer  $n \ge 1$   such that   $f^n(x) \in \mathcal{O}$ for all  $x \in P_f$.
  Take an arbitrary $\gamma \in \Pi$ and  let  $\eta$ be
 a non-peripheral component of $F^{-n}(\gamma)$.  If $\eta \notin \Lambda$, then there would be a $\emph{type P}$ curve segment  of $\eta$, say $\sigma$,
  such that $D(\sigma) \cap P_f \ne \emptyset$. By applying Lemma~\ref{ess} $n$ times,
   we would have some $\emph{type P}$ curve segment of $\gamma$, say $\tau$,
   such that $D(\tau) \cap \mathcal{O} \ne \emptyset$. This implies that $\gamma \in \Sigma$ which contradicts the assumption that $\gamma \in \Pi$.

   \end{proof}

 Now suppose $F$ has an obstruction.   Then by \cite{Pi} $F$ has a  canonical Thurston, say $\Gamma$, which consists of  all homotopy classes of the
 non-peripheral curves whose length go to zero as we iterate the Thurston pull back induced by $F$.   We claim that  $\Gamma \cap \Lambda = \emptyset$.  Let us prove the claim. Suppose  $\Gamma \cap \Lambda \ne \emptyset$. Then
   $\Gamma \cap \Lambda$ must be  $F$-stable by Lemma~\ref{ess}.  Since the length of every
    curve  in $\Gamma \cap \Lambda$ goes to zero as we iterate the Thurston pull back, the transformation matrix   associated to $\Gamma \cap \Lambda$ must have an eigenvalue $\ge 1$.   By the definition of $\Lambda$,  one can deform the curves in  $\Gamma \cap \Lambda$  so that   it is a stable family of $g$ and with the same transformation matrix.  This is a contradiction because $g$ has no obstruction.   This,  together with Lemma~\ref{ev}, implies that $\Gamma \cap \Pi = \emptyset$.  We thus have

   \begin{lema}\label{short}
  If there is an obstruction for $F$, then there must be  one    which  consists  of curves in $\Sigma$  whose length go to zero as we iterate the Thurston pull back induced by $F$.
   \end{lema}

\begin{lema}\label{o-o}
Let $\gamma$ be a non-peripheral curve.  Let $x \in P_f$. Then
$$
\sum_{\eta}N(\Sigma_x(\eta))  \le N(\Sigma_{f(x)}(\gamma))
$$ where the sum is taken over all the non-peripheral components of $F^{-1}(\gamma)$. In particular,
$$
\sum_{\eta}N(\Sigma_x([\eta]))  \le N(\Sigma_{f(x)}([\gamma])).
$$
\end{lema}

\begin{proof}

Let
$$
\Sigma_x = \bigcup_\eta \Sigma_x(\eta)
$$ where the union is taken over all the non-peripheral components of $F^{-1}(\gamma)$.
We may introduce an order in $\Sigma_x$: $\sigma < \sigma'$ if and only if
 $D(\sigma) \subset D(\sigma')$.   Similarly we   introduce an order in  $\Sigma_{f(x)}(\gamma)$ by setting $\tau < \tau'$ if and only if
 $D(\tau) \subset D(\tau')$.

  Now for each $\sigma \in \Sigma_x$, as in the proof of Lemma~\ref{ess}, $F(\sigma) - \Bbb T$ has at least one component  in $\Sigma_{f(x)}(\gamma)$. Let $M(\sigma)$ denote the maximal one among these elements.    It is sufficient to show that
 $$
 \sigma < \sigma' \Longrightarrow M(\sigma) < M(\sigma').
 $$
 But this follows from the polynomial property: as we make $D(\sigma)$ larger, the polynomial image of $D(\sigma)$ will become larger, and therefore, $M(\sigma)$  will become strictly larger. See Figure 7 for an illustration.
\end{proof}
\begin{figure}[!htpb]
  \setlength{\unitlength}{1mm}
  \centering
  \includegraphics[width=100mm]{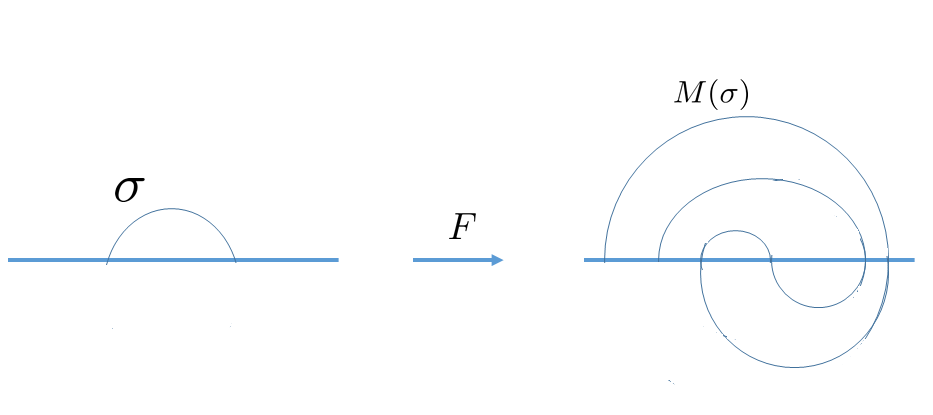}
  \caption{ $\sigma < \sigma' \Longrightarrow M(\sigma) < M(\sigma')$}
  \label{Figure-2}
  \end{figure}

\begin{cor}\label{co}
Suppose $\gamma$ is a non-peripheral curve and $\eta$ is a non-peripheral component of $F^{-1}(\gamma)$. Then for any $x \in P_f$, we have
$$
N(\Sigma_x([\eta]))  \le N(\Sigma_{f(x)}([\gamma]))
$$
\end{cor}
Now let us prove the main theorem. Suppose $F$ has an obstruction and let $\Gamma$ be an obstruction guaranteed by  Lemma~\ref{short}. We may assume that it is an irreducible one. That is, for any $\gamma, \eta \in \Gamma$, there is an $n \ge 1$ such that $\eta$ is homotopic to a component of $f^{-n}(\gamma)$.  Now let us prove $\Gamma$ is a Levy cycle.

Claim:  for each $\gamma \in \Gamma$, there exists exactly one non-peripheral component of $F^{-1}(\gamma)$. Suppose this were not true. Then we would have two non-peripheral components  $\gamma_1 \ne \eta_1$ of $F^{-1}(\gamma)$  and two   sequence:
\begin{equation}\label{fs}
\gamma = \gamma_0 \to \gamma_1 \to \gamma_2 \to \cdots \to \gamma_l \to  \gamma_{l+1} = \gamma
\end{equation}  and
\begin{equation}\label{ss}
\gamma = \eta_0 \to \eta_1 \to \eta_2 \to \cdots \to \eta_k \to \eta_{k+1} =  \gamma,
\end{equation}   where $\gamma_{i+1}$ is homotopic to some component of $f^{-1}(\gamma_i)$, $0 \le i \le l+1$,
 and $\eta_{j+1}$ is homotopic to some component of $f^{-1}(\eta_j)$, $0 \le j \le k$.
Since $\gamma \in \Sigma$, we have some  $x \in \mathcal{O}$ such that $N(\Sigma_x([\gamma])) > 0$. Let  $y \in \mathcal{O}$ such that $x  = F(y)$.
Let $p$ be the period of the $x$. Repeating (\ref{fs}) $p$ times,
$$\gamma_0 \to \gamma_1 \to \gamma_2 \to \cdots \to \gamma_l \to  \gamma_{0} \to \cdots \to \gamma_0 \to \gamma_1 \to \gamma_2 \to \cdots \to \gamma_l \to  \gamma_{0}
$$
Apply  Corollary~\ref{co} to the above sequence, we get $$ N(\Sigma_x([\gamma]) \ge   N(\Sigma_y([\gamma_1])) \ge   N(\Sigma_x([\gamma])),$$ which implies that $$N(\Sigma_y([\gamma_1]))  =    N(\Sigma_x([\gamma])).$$  Similarly, we may repeat (\ref{ss}) $p$ times and then apply  Corollary~\ref{co},  we get
$$
N(\Sigma_x([\gamma]))  \ge  N(\Sigma_y([\gamma_1])) \ge  N(\Sigma_x([\gamma])),
$$ which implies that $$N(\Sigma_y([\eta_1]))  =    N(\Sigma_x([\gamma])).$$  But by Lemma~\ref{o-o}  and $\gamma_1 \ne \eta_1$,  we also have
$$
N(\Sigma_x([\gamma])) \ge N(\Sigma_y([\gamma_1])) +  N(\Sigma_y([\eta_1])).
$$
This implies that $N(\Sigma_x([\gamma])) = 0$. This  contradicts the assumption that $\Sigma_x([\gamma]) > 0$.  The Claim has been proved.
\begin{figure}[!htpb]
  \setlength{\unitlength}{1mm}
  \begin{center}

 \includegraphics[width=60mm]{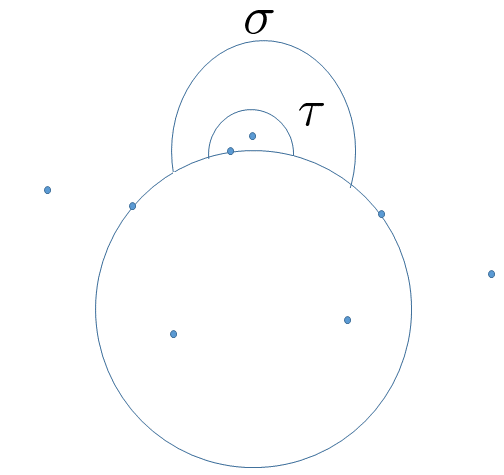}
  \caption{Two homotopic type P curve segments}
  \label{Figure-8}

  \end{center}
  \end{figure}
  Now for any non-peripheral curve $\gamma$, let $K(\gamma)$ denote the number of the $\emph{type P}$ curve segments of $\gamma$ and let
   $$
   K([\gamma]) = \min K(\eta)
   $$ where $\min$ is taken over all the non-peripheral curves which are homotopic to $\gamma$.  From the claim above,  it follows  that
  $\Gamma$ must be a Levy cycle for $F$. Let
   $\Gamma = \{\gamma_1, \cdots, \gamma_n\}$   so that $$K(\gamma_1) =  \min_{1 \le i \le n}  K([\gamma_i])$$ and
    for $1 \le i \le n-1$, $\gamma_{i+1}$ is the unique non-peripheral component of $F^{-1}(\gamma_{i})$.    Since the image of
   a $\emph{type P}$ curve segment contains at least one $\emph{type P}$ curve segment,  we must have
    $$
      K(\gamma_1) =  \cdots =  K(\gamma_n).
    $$
So  all $\emph{type P}$ curve segments of $\gamma_i$, $1 \le i \le n$,  are non-trivial in the sense that   $D(\sigma) \cap P_F \ne \emptyset$.  So for any $\emph{type P}$ curve segment $\sigma$
of $\gamma_{i+1}$, there is a $\emph{type P}$ curve segment $\sigma'$  which is homotopic to $\sigma$ (Figure 8 illustrates  the meaning that two  $\emph{type P}$ curve segments are homotopic),  such that $\sigma'$  is mapped homeomorphically to some $\emph{type P}$ curve segment $\tau$ of $\gamma_i$. We thus
    get a cycle of $\emph{type P}$ curve segments $\{\sigma_i\}$, $1 \le i \le m$ with $n|m$, such that for each $\sigma_i$, there is a $\emph{type P}$ curve segment $\mu_{i+1}$ which is homotopic to $\sigma_{i+1}$ so that $F: \mu_{i+1} \to \sigma_i$ is a homeomorphism. Since $f$ is post-critically finite which  expands the orbifold metric, it follows that $D(\sigma_i)$ contains exactly one point in $P_F$, say $x_i$, which lies in $I_i = D(\sigma_i) \cap \Bbb T$, and moreover, $\{x_i\}$ is a periodic cycle.  But one may then deform  each  $D(\sigma_i)$  into a small neighborhood of $x_i$. In this way $\Gamma = \{\gamma_i\}$ becomes into a Levy cycle of $g$.
 This is impossible.  The proof of the main theorem is completed.

\medskip

  $\bold{Acknowledgements.}$ The author would like to thank    Fei Yang who provides all the computer generating pictures in the paper.

\end{document}